\documentclass[leqno]{article}
\usepackage{amsmath,amssymb,amsthm,amscd,dsfont}
\numberwithin{equation}{section} \allowdisplaybreaks
\begin{document}
\newtheorem{theorem}{Theorem}[section]
\newtheorem{defin}{Definition}[section]
\newtheorem{prop}{Proposition}[section]
\newtheorem{corol}{Corollary}[section]
\newtheorem{lemma}{Lemma}[section]
\newtheorem{rem}{Remark}[section]
\newtheorem{example}{Example}[section]
\title{Soldered tensor fields of normalized submanifolds}
\author{{\small by}\vspace{2mm}\\Izu Vaisman}
\date{{\footnotesize Dedicated to the centenary of the Mathematical
Seminar ``Al. Myller", Jassy, Romania and to Acad. Prof.
Constantin Corduneanu on his 80-eth anniversary}}
\maketitle
{\def\thefootnote{*}\footnotetext[1]%
{{\it 2000 Mathematics Subject Classification: 53C40} .
}}
\begin{center} \begin{minipage}{12cm}
A{\footnotesize BSTRACT. Soldered forms, multivector fields and
Riemannian metrics were studied in our earlier paper
\cite{V1}. In particular, it was shown that a Riemannian
submanifold is totally geodesic iff the metric is soldered to the
submanifold. In the present paper, we discuss general, soldered
tensor fields. In particular, we prove that the almost complex
structure of an almost K\"ahler manifold is soldered to a
submanifold iff the latter is an invariant, totally geodesic
submanifold.}
\end{minipage}
\end{center}
\vspace*{5mm}
\section{Introduction}
In the present paper, all the manifolds, mappings, bundles, tensor
fields, etc. are differentiable of class $C^\infty$ and we use the
standard notation of Differential Geometry, including the Einstein
summation convention. The reader may consult \cite{KN} for all the
differential geometric notion and results that are used in the
paper.

If $N^n$ is a submanifold of $M^m$ (indices denote dimension),
then a {\it normalization} of $N$ by a {\it normal bundle} $\nu N$
is a splitting
\begin{equation} \label{split} TM|_N=TN\oplus\nu N
\end{equation} ($T$ denotes tangent bundles). A submanifold
endowed with a normalization is called a {\it normalized
submanifold} and a vector field $X$ on $M$ is tangent or normal to
$N$ if $X|_N$ belongs to $TN,\nu N$, respectively. The best known
case is that of a {Riemannian normalization} $\nu N=T^{\perp_g}N$,
where $g$ is a Riemannian metric on $M$. In fact, given an
arbitrary normalization, it is easy to construct metrics $g$ such
that the normalization is $g$-Riemannian. Similarly, if $N$ is a
symplectic submanifold of a symplectic manifold $(M,\omega)$, $\nu
N=T^{\perp_\omega}N$ defines the {\it symplectic normalization}.
Another interesting example is that of a submanifold $N$ such
that, $\forall x\in N$, $T_xM=T_xN\oplus T_x\mathcal{F}$, where $
\mathcal{F}$ is a foliation  of $M$; then we may take $\nu
N=T\mathcal{F}|_N$.

In our earlier paper \cite{V1} we discussed differential forms,
multivector fields and Riemannian metrics that have a special kind
of contact with a normalized submanifold; these were said to be
{\it soldered} to the submanifold. In particular, it was shown
that a Riemannian submanifold is totally geodesic iff the metric
is soldered to the submanifold and that a submanifold of a Poisson
manifold is a (totally) Dirac submanifold iff there exists a
normalization such that the Poisson bivector field is soldered to
the submanifold with respect to this normalization.

In the present paper we give a general definition for the notion
of soldering of an arbitrary tensor field and we consider an
obstruction to soldering, which, essentially, is a generalization
of the second fundamental form of a Riemannian submanifold. We
establish some formulas for the calculation of this obstruction
and get corresponding applications. In particular, we prove that
the almost complex structure $J$ of an almost K\"ahler manifold is
soldered to a submanifold iff the latter is a $J$-invariant,
totally geodesic submanifold.
\section{Soldered tensor fields}
Let $(N^n,\nu N)$ be a normalized submanifold of $M^m$ and let
$\iota:N\subseteq M$ be the corresponding embedding.

First, we exhibit some adequate, local coordinates around the
points of $N$. Let $\sigma:W\rightarrow N$ be a tubular
neighborhood of $N$ such that $\forall x\in N$, $T_x(W_x)=\nu_x N$
($W_x$ is the fiber of $W$ and $\nu_xN$ is the fiber of $\nu N$ at
$x$). For every point $x\in N$ there exists a
$\sigma$-trivializing neighborhood $U$ with coordinates $(x^a)$
$(a,b,c,...= 1,...,m-n)$ around $x$ on the fibers of $\sigma$,
such that $x^a|_{N\cap U}=0$, and coordinates $(y^u)$
$(u,v,w,...=m-n+1,...,m)$ around $x$ on $N\cap U$. We say that
$(x^a,y^u)$ are {\it adapted local coordinates}.

Then,
\begin{equation} \label{adapted} TN|_{N\cap U}=span\left\{\left.\frac{\partial}
{\partial y^u}\right|_{x^a=0}\right\},\;\; \nu N|_{N\cap U}
=span\left\{\left.\frac{\partial} {\partial
x^a}\right|_{x^a=0}\right\}\end{equation} and the transition
functions between systems of adapted local coordinates have the
local form
\begin{equation}
\label{transition} \tilde x^a=\tilde x^a(x^b,y^v),\;\tilde
y^u=\tilde y^u(y^v),
\end{equation} where \begin{equation} \label{transcond}
\left.\frac{\partial\tilde x^a}{\partial y^v}\right|_{x^b=0}=0,
\;\frac{\partial\tilde y^u}{\partial x^b}\equiv0.
\end{equation}
Furthermore, (\ref{split}) implies
\begin{equation} \label{dualsplit}
T^*M|_N=T^*N\oplus\nu^*N,\end{equation} and
\begin{equation}\label{dualspan}
\begin{array}{c} T^*N=ann(\nu N)= span\{dy^u|_{x^a=0}\},\vspace{2mm}\\
\nu^*N=ann(TN)=span\{dx^a|_{x^a=0}\} \end{array}
\end{equation} ($ann$ denotes annihilator spaces).

Now we give the following general definition.
\begin{defin}\label{formsold} {\rm A tensor field
$A\in\mathcal{T}^p_q(M)$ (where $\mathcal{T}$ denotes a space of
tensor fields) is {\it soldered} to the normalized submanifold
$(N,\nu N)$ if for any normal vector field $X\in\Gamma TM$ of $N$
one has
\begin{equation}\label{eqsoldA} (L_XA)_x(Y_1,...,Y_q,\xi_1,...,\xi_p)=0,
\end{equation} for any $x\in N$ and any arguments
$Y_1,...,Y_q\in T_xN,\xi_1,...,\xi_p\in ann\,\nu_x N$.}
\end{defin}

In (\ref{eqsoldA}) $L$ denotes the Lie derivative and it turns out
that (\ref{eqsoldA}) is a combination of algebraic and
differential conditions. Indeed, we have
\begin{prop}\label{propalogcond} If the tensor field
$A\in\mathcal{T}^p_q(M)$ is soldered to the normalized submanifold
$(N,\nu N)$, then, for any fixed vectors $Y_1,...,Y_q\in T_xN$ and
covectors $\xi_1,...,\xi_p\in ann\,\nu_x N$, the following algebraic
conditions must hold:

1) the $1$-forms $\alpha_i\in T^*M|_N$, $i=1,...,q$, defined by
\begin{equation}\label{defalpha}\alpha_i(V)=A|_N(Y_1,...,Y_{i-1},V,
Y_{i+1},...,Y_q,\xi_1,...,\xi_p),\;\;V\in TM|_N,\end{equation}
belong to $ann\,\nu N$;\\

2) the vector fields $Z_j\in TM|_N$, $j=1,,,,p$, defined by
\begin{equation}\label{defZ}
Z_j(\gamma)=A|_N(Y_1,...Y_q,\xi_1,...,,\xi_{j-1},\gamma,
\xi_{j+1},...,\xi_p),\;\;\gamma\in T^*M|_N,\end{equation} are tangent to $N$.
\end{prop}
\begin{proof} Consider the general formula
\begin{equation}\label{Lgenform}(L_{\varphi V}A)(Y_1,...,Y_q,\xi_1,...,\xi_p)=
\varphi(L_{V}A)(Y_1,...,Y_q,\xi_1,...,\xi_p)\end{equation}
$$-\sum_{i=1}^q(Y_i\varphi)\alpha_i(V)-\sum_{j=1}^p\xi_j(V)(Z_j\varphi),$$
where $\alpha_i(V)$ and $Z_j\varphi=d\varphi(Z_j)$ are defined by
(\ref{defalpha}), (\ref{defZ}), respectively; the formula holds
for arbitrary arguments (not necessarily related to $N$) and for
any function $\varphi\in C^\infty(M)$. Conditions 1), 2), follow
from (\ref{Lgenform}) by taking $V=fX+lY$ where $l|_N=0$; this
vector field is normal to $N$ again, therefore, it also satisfies
(\ref{eqsoldA}). \end{proof}
\begin{prop}\label{condsoldloc} The tensor field $A\in\mathcal{T}^p_q(M)$
is soldered to the normalized submanifold $N$ iff the local
components of $A$ with respect to adapted coordinates satisfy the
conditions
\begin{equation}\label{12local}
A_{u_1,...,u_{i-1},a,u_{i+1},...,u_q}^{v_1,...,v_p}(0,y^w)=0,\;
A_{u_1,...,u_q}^{v_1,...,v_{j-1},a,v_{j+1},...,v_p}(0,y^w)=0
\end{equation} and
\begin{equation}\label{eqsoldloc}\left.
\frac{\partial A^{v_1,...,v_p}_{u_1,...,u_q}}{\partial
x^a}\right|_{x^b=0}=0.\end{equation}\end{prop}
\begin{proof} Using the bases (\ref{adapted}), (\ref{dualspan}),
we see that conditions 1), 2) of Proposition \ref{formsold} are
equivalent to (\ref{12local}) and (\ref{eqsoldloc}) is
(\ref{eqsoldA}) expressed for $X=\partial/\partial x^a$.
Conversely, using formula (\ref{Lgenform}), it is easy to derive
(\ref{eqsoldA}) from (\ref{eqsoldloc}) and the algebraic
conditions 1), 2).\end{proof}

In the case of either a differential form or a multivector field
formulas (\ref{12local}), (\ref{eqsoldloc}) reduce to the
conditions for soldering forms and multivector fields given in
\cite{V1}.
\begin{example}\label{exfol} {\rm Assume that there exists a
foliation $\mathcal{F}$ of $M$ such that $\nu N=T\mathcal{F}|_N$
is a normalization of $N$. A tensor field $A\in\mathcal{T}^p_q(M)$
is said to be {\it projectable} or {\it foliated} if for any local
quotient manifold $Q_U=U/\mathcal{F}\cap U$ ($U$ is an open
neighborhood in $M$ where $\mathcal{F}$ is simple) there exists a
tensor field $A'\in\mathcal{T}^p_q(Q_U)$ that is $\pi$-related to
A ($\pi$ is the natural projection $U\rightarrow Q_U$). Let
$(y^u,x^a)$ be local coordinates such that the local equations of
the leaves are $y^u=const.$ (In particular, around points $x\in N$
we may use $N$-adapted local coordinates.) Then, it is easy to see
that $A$ is projectable iff it has a local expression of the
following form
\begin{equation}\label{Aproj} A= dy^{u_1}\otimes...\otimes
dy^{u_q}\otimes
[A_{u_1...u_q}^{v_1...v_p}(y)\frac{\partial}{\partial y^{v_1}}
\otimes...\otimes\frac{\partial}{\partial y^{v_p}}\end{equation} $$+
A_{u_1...u_q}^{a_1v_2...v_p}(x,y)\frac{\partial}{\partial x^{a_1}}
\otimes\frac{\partial}{\partial
y^{v_2}}...\otimes\frac{\partial}{\partial y^{v_p}} + ... +
A_{u_1...u_q}^{a_1a_2...a_p}(x,y)\frac{\partial}{\partial x^{a_1}}
...\otimes\frac{\partial}{\partial x^{a_p}}].$$ Formula
(\ref{Aproj}) shows that the projectable tensor fields are
characterized by the following global properties:

\hspace*{1.5mm}(i)\hspace*{5mm}
$A\in[\otimes^q(ann\,T^*\mathcal{F})]\otimes[\otimes^pTM]$,

(ii)\hspace*{5mm} $\forall X\in T\mathcal{F}$ one has $L_XA\in
T\mathcal{F}\otimes[\otimes^q(ann\,T^*\mathcal{F})]\otimes[\otimes^{p-1}TM]$.\\
Accordingly, we see that an $ \mathcal{F}$-projectable tensor field
$A$ is soldered to the submanifold $N$ iff the algebraic condition
2) of Proposition \ref{formsold} is satisfied. In particular, a
totally covariant, foliated tensor field necessarily is soldered to
any local transversal submanifold $N^n$ of the foliation
$\mathcal{F}^{m-n}$.}\end{example}
\begin{defin}\label{algadapt} {\rm A tensor field
$A\in\mathcal{T}^p_q(M)$ that satisfies the algebraic conditions 1),
2) of Proposition \ref{formsold} (equivalently, satisfies
(\ref{12local})) will be called {\it algebraically adapted to
$N$}.}\end{defin}
\begin{prop}\label{form2} If $A\in\mathcal{T}^p_q(M)$ is
algebraically adapted to $N$, the morphism $w_A:\nu N\rightarrow
\mathcal{T}^p_q(N)$ defined by
\begin{equation}\label{defw} w_A(\bar
X)(Y_1,...,Y_q,\xi_1,...,\xi_p)=L_{
X}A(Y_1,...,Y_q,\xi_1,...,\xi_p)|_N,\end{equation} where
$Y_1,...,Y_q\in\Gamma TN$, $\xi_1,...,\xi_p\in\Gamma(ann\,\nu N)$,
$\bar X\in\Gamma\nu N$ and $X$ is a vector field on $M$ with the
restriction $\bar X$ to $N$, is independent of the choice of the
extension $X$ of $\bar X$.
\end{prop} \begin{proof} Since $A$ is algebraically $N$-adapted,
formula (\ref{Lgenform}) yields
\begin{equation}\label{invarianta}(L_{fX+lZ}A)(Y_1,...,Y_q,\xi_1,...,\xi_p)|_N=
f(L_{X}A)(Y_1,...,Y_q,\xi_1,...,\xi_p)|_N,\end{equation} for any
functions $f,l\in C^\infty(M)$ such that $l|_N=0$, any vector
field $X$ normal to $N$ and any vector field $Z$ on $M$. The case
$Z=0$ shows that $w_A$ is $C^\infty(N)$-linear in $\bar{X}$. The
case $f=1$ shows that $w_A(\bar{X})$ is independent of the choice
of the extension $X$ of $\bar{X}$, since, using coordinate
expressions, it easily follows that two extensions $X_1,X_2$ are
related by an equality of the form $X_2=X_1+\sum l_kZ_k$ where the
functions $l_k$ vanish on $N$.
\end{proof}
\begin{defin}\label{defobstr} {\rm The morphism $w_A$ will be called
the {\it soldering obstruction} of the algebraically $N$-adapted
tensor field $A$.}\end{defin}

The name is motivated by the fact that if $w_A=0$ then $A$ is
soldered to $N$. Notice that $w_A(\bar X)$ has the same symmetries
like $A$.
\section{Applications}
In this section we consider only Riemannian normalizations,
therefore, $M$ is endowed with a Riemannian metric $g$ and $\nu
N=T^{\perp_g}N$. All the vector fields denoted by $Y$ are tangent to
$N$ and all the vector fields denoted by $X$ are normal to $N$. The
reader is asked to pay attention to the situations where
calculations take place only along $N$. In \cite{V1}, as a
consequence of the Gauss-Weingarten formulas \cite{KN}, we proved
that the soldering obstruction of the metric $g$ is
\begin{equation}\label{invptg}
w_g(\bar X)(Y_1,Y_2)=(L_Xg)(Y_1,Y_2)|_N=-2g(\beta(Y_1,Y_2),\bar
X)\end{equation} where $X|_N=\bar X\in\Gamma\nu N$ and $\beta$ is
the second fundamental form of $N$. Accordingly, the metric is
soldered to a submanifold $N$ iff $N$ is a totally geodesic
submanifold of $M$.

For more applications we compute the soldering obstruction of a
tensor field $A$ of type $(1,1)$. From (\ref{12local}) it follows
that $A$ is algebraically adapted to $(N,\nu N)$ iff both $TN$ and
$\nu N$ are invariant by the endomorphism $A$ and we shall assume
that this condition holds. Together with the soldering invariant of
$A$ we define the bilinear {\it soldering form}
$\sigma_A(Y_1,Y_2)\in\nu N$ given by
\begin{equation}\label{soldform}
g(\sigma_A(Y_1,Y_2),\bar X)=w_A(\bar
X)(Y_1,\flat_gY_2).\end{equation} Of course, $A$ is soldered to
$(N,T^{\perp_g}N)$ iff $\sigma_A=0$.
\begin{prop}\label{symsold} Assume that the operator $A$ is either
symmetric or skew-symmetric with respect to $g$. Then, $\sigma_A$ is
symmetric, respectively, skew-symmetric iff $A$ is symmetric,
respectively, skew-symmetric with respect to the second fundamental
form $\beta$ of the submanifold $N$ of $(M,g)$.\end{prop}
\begin{proof} The assumed symmetry property is
\begin{equation}\label{symA}
g(AV_1,V_2)=\pm g(V_1,AV_2),\hspace{3mm} \forall V_1,V_2\in\Gamma
TM.\end{equation} If we take the Lie derivative $L_X$ of this
equality, modulo the equality itself, and use (\ref{invptg}), we
get
$$g(\sigma_A(Y_1,Y_2)\mp\sigma_A(Y_2,Y_1),\bar X)=
2g(\beta(AY_1,Y_2)\mp\beta(Y_1,AY_2),\bar X),$$ whence the
conclusion.\end{proof}

The following proposition expresses the soldering form of a
$g$-(skew)-symmetric $(1,1)$-tensor field $A$ in terms of the
Levi-Civita connection $\nabla$ of $g$ and the second fundamental
form $\beta$ of the submanifold $N$.
\begin{prop}\label{legAR} Assume that the algebraically
$N$-adapted tensor field $A\in\mathcal{T}_1^1(M)$ satisfies
{\rm(\ref{symA})}. Then, the following formula, where the sign in
the right hand side is opposite to the sign in {\rm(\ref{symA})},
holds:
\begin{equation}\label{legAR2}\begin{array}{l}
g(\sigma_A(Y_1,Y_2),\bar X)=g(\nabla_{\bar
X}A(Y_1),Y_2)\vspace*{2mm}\\ +g(\beta(AY_1,Y_2)\mp
\beta(Y_1,AY_2),\bar X).\end{array}\end{equation}
\end{prop}
\begin{proof} We prove the equality
at every fixed point $x\in N$. During the calculations, we extend
the vectors $Y_1(x),Y_2(x),\bar X(x)$ to vector fields $\tilde
Y_1,\tilde Y_2,X$ on $M$ that are tangent, respectively, normal to
$N$. Since the final result is independent of the choice of the
extension, we may use local, adapted coordinates and take
\begin{equation}\label{extensiiaux}\tilde
Y_1=\mu^u_1\frac{\partial}{\partial y^u}\,,\tilde
Y_2=\mu^u_2\frac{\partial}{\partial y^u}\, ,
X=\xi^a\frac{\partial}{\partial x^a},\end{equation} where
$\mu^u_1,\mu^u_2,\xi^a$ are constant (namely, the components of
$Y_1(x),Y_2(x),\bar X(x)$ at the fixed point $x$). From the
equality
$$L_X(A\tilde Y_1)=[X,A\tilde Y_1]=\nabla_X(A\tilde Y_1)-\nabla_{A\tilde
Y_1}X,$$ we get
$$[(L_XA)(Y_1)]_x=[A\nabla_{X}\tilde Y_1-\nabla_{AY_1}X
+(\nabla_{X}A)(Y_1)]_x.$$ In this result we may replace
$\nabla_{A_xY_1(x)} X=-W_{\bar X(x)}\tilde{Y}_1+D_{Y_1(x)}X$,
where $W$ is the Weingarten operator of $N$ and $D$ is the
connection induced by $\nabla$ in $\nu N$. Then, using also
(\ref{symA}), we get
$$[g(\sigma_{A}(Y_1,Y_2),\bar X)]_x=[g(\beta(AY_1,Y_2),\bar X)\pm
g(\nabla_{\bar X}\tilde Y_1,AY_2)+g((\nabla_{\bar
X}A)(Y_1),Y_2)]_x.$$ But, $\nabla_{\bar X}\tilde Y_1=
\nabla_{Y_1}X+[X,\tilde Y_1]$, and the last bracket
vanishes for the chosen extensions (\ref{extensiiaux}).
Accordingly, $$[g(\nabla_{\bar X}\tilde
Y_1,AY_2)]_x=[g(\nabla_{Y_1}X,AY_2)=g(-W_{\bar
X}Y_1+D_{Y_1}X,AY_2)]_x$$
$$=-[g(\beta_x(Y_1,AY_2),X)]_x$$ and (\ref{legAR2}) follows.
\end{proof}
\begin{corol}\label{corolAR} Assume that $A$ is parallel with respect
to the Levi-Civita connection of $g$. Then, if $A$ is $g$-symmetric
$\sigma_A=0$ and if $A$ is $g$-skew-symmetric
$\sigma_A(Y_1,Y_2)=2\beta(AY_1,Y_2)$, where $\beta$ is the
Riemannian, second fundamental form of $N$ in $M$.
\end{corol}
\begin{proof}
The Gauss equation $$\nabla_{Y_1}(A\tilde Y_2)=\nabla'_{Y_1}(A
Y_2)+\beta(AY_1,Y_2),$$ where $\nabla'$ is the induced Levi-Civita
connection on $N$, implies that, if $\nabla A=0$, then
$\beta(AY_1,Y_2)=A\beta(Y_1,Y_2)$. Similarly,
$\beta(Y_1,AY_2)=A\beta(Y_1,Y_2)$. Inserting $\nabla A=0$ and the
previous results for $\beta$ in (\ref{legAR2}) we get the
announced results.\end{proof}

Another nice formula is given by
\begin{prop}\label{AsiNij} Let $A$ be either a $g$-symmetric
or a $g$-skew-symmetric $(1,1)$-tensor field on $(M,g)$ that is
algebraically adapted to the submanifold $N$ and let
$\mathcal{N}_A$ be its Nijenhuis tensor. Then, one has
\begin{equation}\label{sigmaNij}
g(\sigma_A(Y_1,Y_2),A\bar X)=\pm g(\sigma_A(Y_1,AY_2),\bar
X)+g(\mathcal{N}_A(X,Y_1),Y_2).\end{equation}
\end{prop}
\begin{proof} With the notation in the proof of Proposition \ref{legAR},
if $x\in N\subseteq M$, one has the following expression of the
Nijenhuis tensor
\begin{equation}\label{NijcuL}
[(\mathcal{N}_A)(
X,\tilde{Y}_1)]_x=[(L_{AX}A)(\tilde{Y}_1)-A(L_{X}A)(\tilde{Y}_1)]_x.\end{equation}
The required result follows by taking the $g$-scalar product of
the previous equality by $Y_2$.\end{proof}

For a concrete application, let $(M,J,g)$ be an almost Hermitian
manifold, which means that the almost complex structure $J$ is
$g$-skew-symmetric. The tensor field $J$ is algebraically adapted
to the submanifold $N$ iff $N$ is $J$-invariant, which we shall
assume hereafter. If $(M,J,g)$ is a K\"ahler manifold and $N$ is a
complex submanifold, Corollary
\ref{legAR} gives $\sigma_J(Y_1,Y_2)=2\beta(JY_1,Y_2)$ and we see
that $J$ is soldered to $(N,T^{\perp_g}N)$ iff $N$ is a totally
geodesic submanifold.

We shall extend this result to almost K\"ahler manifolds. For any
almost Hermitian manifold $(M,J,g)$ one has the K\"ahler form
$\Omega(Y_1,Y_2)=g(JY_1,Y_2)$.
\begin{prop}\label{soldJOmega} The soldering form of the almost
complex structure $J$ is related to the K\"ahler form $\Omega$ by
means of the formula
\begin{equation}\label{leg3forme2}
g(\sigma_J(Y_1,Y_2),\bar X)=2g(\beta(JY_1,Y_2),\bar
X)+d\Omega(\bar X,Y_1,Y_2).\end{equation}
\end{prop}
\begin{proof} From the definition of $\Omega$ we
get
$$(L_X\Omega)(Y_1,Y_2)=(L_Xg)(JY_1,Y_2)+g(L_XJ(Y_1),Y_2),$$
which, for $\bar X\in\nu N$ for $Y_1,Y_2\in TN$ becomes
\begin{equation}\label{leg3forme}g(\sigma_J(Y_1,Y_2),\bar X)=(L_X\Omega)(Y_1,Y_2)
-g(\sigma_g(JY_1,Y_2),\bar X).\end{equation} Notice that $\Omega$
is algebraically compatible with $(N,T^{\perp_g})$, therefore,
$$(L_X\Omega)(Y_1,Y_2)=g(\sigma_\Omega(Y_1,Y_2),\bar X).$$
Since it is easy to check that for the involved arguments one has
$(di(X)\Omega)(Y_1,Y_2)=0$, by using $L_X=di(X)+i(X)d$ in
(\ref{leg3forme}), we get
\begin{equation}\label{leg3forme'}g(\sigma_J(Y_1,Y_2),\bar X)=d\Omega(\bar X,Y_1,Y_2)
-g(\sigma_g(JY_1,Y_2),\bar X).\end{equation} In view of
(\ref{invptg}), formula (\ref{leg3forme'}) is the same as the one
required by the proposition.\end{proof}
\begin{corol}\label{JsoldO} The almost complex structure $J$ is
soldered to the $J$-invariant submanifold $N$ iff the second
fundamental form of $N$ is given by the formula
\begin{equation}\label{f2ptJsold}
g(\beta(JY_1,Y_2),X)=-\frac{1}{2}d\Omega(X,Y_1,Y_2).\end{equation}
\end{corol}

Then, since an almost K\"ahler manifold is characterized by the
property $d\Omega=0$, we get the main application:
\begin{prop}\label{almKtgeod} If $(M,J,g)$ is an almost K\"ahler manifold,
the almost complex structure $J$ is soldered to the submanifold $N$
iff $N$ is a $J$-invariant, totally geodesic submanifold of
$M$.\end{prop}
\hspace*{7.5cm}{\small \begin{tabular}{l} Department of
Mathematics\\ University of Haifa, Israel\\ E-mail:
vaisman@math.haifa.ac.il \end{tabular}}
\end{document}